\documentclass[onefignum,onetabnum]{siamonline190516}

\usepackage{amsfonts,amssymb,amsmath}
\usepackage{multicol}
\usepackage{graphicx}
\usepackage{epstopdf}
\usepackage{algorithmic}
\ifpdf
  \DeclareGraphicsExtensions{.eps,.pdf,.png,.jpg}
\else
  \DeclareGraphicsExtensions{.eps}
\fi

\usepackage{tikz}

\usepackage{enumitem}
\setlist[enumerate]{leftmargin=.5in}
\setlist[itemize]{leftmargin=.5in}

\newsiamremark{remark}{Remark}
\newsiamremark{hypothesis}{Hypothesis}
\crefname{hypothesis}{Hypothesis}{Hypotheses}
\newsiamthm{claim}{Claim}

\headers{Algebraic Experimental Design}{Dimitrova, Fredrickson, Rondoni, Stigler, Veliz-Cuba}

\title{Algebraic Experimental Design: Theory and Computation\thanks{Submitted to the editors \today.
\funding{Dimitrova, Fredrickson, and Rondoni were partially supported by NSF Award DMS-1419038; Stigler was partially supported by NSF Award DMS-1720335; Veliz-Cuba was partially supported by the Simons Foundation grant 516088.}}}

\author{
Elena S. Dimitrova%
\thanks{Department of Mathematics, California Polytechnic State University, San Luis Obispo, CA 93407 
(\email{edimitro@calpoly.edu}).} 
\and 
Cameron H. Fredrickson%
\thanks{Department of Mathematics, California Polytechnic State University, San Luis Obispo, CA 93407 
(\email{chfredri@calpoly.edu}).}
\and 
Nicholas A. Rondoni%
\thanks{Department of Applied Mathematics, University of California, Santa Cruz, CA 95064 
(\email{nrondoni@ucsc.edu}).}
\and 
Brandilyn Stigler%
\thanks{Department of Mathematics, Southern Methodist University, Dallas, TX 75275 
(\email{bstigler@smu.edu}).}
\and 
Alan Veliz-Cuba%
\thanks{Department of Mathematics, University of Dayton, Dayton, OH 45469 
(\email{avelizcuba1@udayton.edu}).}
}

\usepackage{amsopn}

\def\D{\mathcal{D}}
\def\F{\mathbb{F}}
\def\G{\mathcal{G}}

\def\M{\mathcal{M}}

\def\p{\mathfrak{p}}
\def\s{\mathbf{s}}
\def\m2{\textit{Macaulay2}}

\renewcommand{\>}{\rangle}
\def\Not{\overline}

\DeclareMathOperator{\Mod}{Mod}
\DeclareMathOperator{\supp}{supp}

\def\F{\mathbb F}

\newtheorem{example}[theorem]{Example}

\usepackage{xcolor}
\newcounter{comments}
\definecolor{Red}{rgb}{0.8,0,0}
\definecolor{Green}{rgb}{0,0.7,0}
\definecolor{Blue}{rgb}{.2,.2,1}

\ifpdf
\hypersetup{
  pdftitle={Algebraic Experimental Design},
  pdfauthor={Dimitrova, Fredrickson, Rondoni, Stigler, Veliz-Cuba}
}
\fi

\begin{document}

\maketitle

\begin{abstract}
Over the past several decades, algebraic geometry has provided innovative approaches to biological experimental design that resolved theoretical questions and improved computational efficiency. However, guaranteeing uniqueness and perfect recovery of models are still open problems. In this work we study the problem of uniqueness of wiring diagrams. We use as a modeling framework polynomial dynamical systems and utilize the correspondence between simplicial complexes and square-free monomial ideals from Stanley-Reisner theory to develop theory and construct an algorithm for identifying input data sets $V\subset \mathbb F_p^n$ that are guaranteed to correspond to a unique minimal wiring diagram regardless of the experimental output. We apply the results on a tumor-suppression network mediated by epidermal derived
growth factor receptor and demonstrate how careful experimental design decisions can lead to a unique minimal wiring diagram identification. One of the insights of the theoretical work is the connection between the uniqueness of a wiring diagram for a given $V\subset \mathbb F_p^n$ and the uniqueness of the reduced Gr\"obner basis of the polynomial ideal $I(V)\subset \mathbb F_p[x_1,\ldots, x_n]$. We discuss existing results and introduce a new necessary condition on the points in $V$ for uniqueness of the reduced Gr\"obner basis of $I(V)$. 
These results also point to the importance of the relative proximity of the experimental input points on the number of minimal wiring diagrams, which we then study computationally. We find that there is a concrete heuristic way to generate data that tends to result in fewer  minimal wiring diagrams.

\end{abstract}



\section{Introduction}


The abundance of numerous substantial data sets from laboratory experiments and myriad diverse methods for modeling and analysis render network inference  a critical component of  systems biology research; for a recent example, see~\cite{grand}. A vital process linked to inference is experimental design, which optimizes data generation and collection for effective prediction of network structure.  While traditional experimental design is rooted in statistical methods~\cite{litwin}, algebraic geometry has offered innovative approaches to experimental design \cite{wynn,he-unique-gbs}. In fact a fractional factorial design can be viewed as a set of $n$-tuples over a finite field
and a special class of discrete models called \emph{polynomial dynamical systems} can be used to capture all models which fit the design points for a network with $n$ nodes.

Associated to a polynomial dynamical system is a directed graph called the \emph{wiring diagram}, which encodes the topology (connectivity) of the network. While the wiring diagram represents only a static picture of the network, knowledge of the connectivity is crucial for studying network robustness, regulation, and control strategies in order to develop, for example, therapeutic interventions~\cite{tan2013, Wang:2013aa} and drug delivery strategies~\cite{yousefi2012,Lee:2012aa}, or to understand the mechanisms for the spread of an infectious disease~\cite{Madrahimov:2013dq,PMID:20478257}. Moreover, it has been demonstrated that the role of network connectivity goes beyond static properties and can in fact dictate certain dynamical properties and be used for their control~\cite{jarrah2010dynamics,campbell,veliz2011reduction,zamal,wu,albert,sontag2008effect,murrugarra15, murrugarra19}. 

In this work, we will develop theory and algorithms for experimental design which reduce the size of the space of possible wiring diagrams. 
The central object of study is a \emph{minimal set} for a node $x$,  that is a set of variables representing the incoming edges to $x$ in the wiring diagram.  Each minimal set, or \emph{minset} for short, has the property that there exists a polynomial in those variables that fits the data (design points) and there is no such polynomial for any proper subset.
Specifically, we aim to find properties on input-output data $(V,T)$ that guarantees that it has a unique minimal set. In this way, we contribute a number of distinct results. 
When only the design points, referred to as \emph{inputs},~$V$ are known, we prove
 a necessary and sufficient condition on $V$ (Theorem~\ref{unique-minset}); a necessary condition on $V$ (Theorem~\ref{diagonal}); and a sufficient condition on $V$ (Corollary~\ref{ugb-minset}).
Each of these conditions on $V$ guarantees that for any corresponding output assignment~$T$, the input-output data set $(V,T)$ has a unique minimal set.  Furthermore, when both inputs~$V$ and outputs~$T$ are known, we provide a sufficient condition on polynomial functions which fit  $(V,T)$ in Theorem~\ref{unique-nf-minset}. 

In parallel, this work has uncovered interesting results for ideals of points.  While it is known that for every monomial order $\prec$ there is a unique reduced Gr\"obner basis $G_\prec$ for $I(V)$, there are cases when the Gr\"obner basis is the same across all monomial orders: that is, there exists a generating set $G$ for $I(V)$ such that for all monomial orders $\prec$ the associated reduced Gr\"obner basis $G_\prec = G$.  In this case we say that \emph{$I(V)$ has a unique reduced Gr\"obner basis for all monomial orders.}  We prove a necessary condition on fixed inputs $V$ (Corollary~\ref{ugb-diag-free}); a necessary condition on arbitrary outputs $T$  (Corollary~\ref{ugb-minset}); and a necessary and sufficient condition on polynomial functions which fit $(V,T)$ for any output $T$ (Corollary~\ref{nf-ugb}).

In an effort to provide guidance for designing experiments, we performed computational experiments that suggest the following rubric: having data with small Hamming distance between points results in fewer minsets than data with large Hamming distance between points.  Moreover we  provide computational evidence that design points generated using a \emph{small-distance scheme} result in fewer minsets than randomly generated points.


The paper is organized as follows. We provide the relevant background in \Cref{sec:background}.  Theoretical results are in \Cref{sec:main}, while computational results are in \Cref{sec:experiments}.  We close with a discussion in \Cref{sec:conclusions}.

\section{Background}
\label{sec:background}

Much of the language in this section is taken from \cite{macauley-stigler}.

Discrete models have been used extensively and there is evidence
that they provide a good framework for a variety of applications, e.g. \cite{davidson,
albert, thomas91, laubenbacher04, dimitrova-zardecki}. Such models are collections of
functions defined over a finite state set $X$ and can be described
using polynomials when the state set size is constrained to a power
of a prime. In the latter case, discrete models are often referred
to as polynomial models and can be written as $n$-tuples of
polynomial functions, one for each node in the network, i.e. $f=(f_1,\ldots,f_n): X^n\to X^n$, where $f_i:X^n\to X$ is a polynomial which determines the behavior of node (variable) $x_i$.  Examples of
polynomial models are Boolean networks ($X=\F_2$) and more generally
\emph{polynomial dynamical systems} (PDSs) over finite fields
($X=\F_p$).

Specifically a \emph{polynomial dynamical system} over $F=\F_p$ is a polynomial map $f:F^n\rightarrow F^n$ where $f=(f_1,\ldots,f_n)$ and each coordinate function $f_i:F^n\rightarrow F$ is a polynomial in $F[x_1,\ldots ,x_n]$. We say that $f$ \emph{fits} the input-output data $D=\{(s_1,t_1),\ldots ,(s_m,t_m)\}\subset F^n\times F^n$ if $f(s_j)=t_j$ for each $1\leq j\leq m$. 

The \emph{monomials} or \emph{terms} of a
polynomial model represent interactions among the nodes in a
network, whereas the coefficient of a monomial can be interpreted as the
strength or weight of the associated interaction.
The \emph{support} of a polynomial $f\in k[x_1,\ldots,x_n]$, denoted $supp(f)$, is the collection of variables that appear in $f$. 

\begin{definition}\label{wd}

A \emph{wiring diagram} of a PDS $f=(f_1,\ldots,f_n)$ is a directed graph $W=(L,E)$ where $|L|=n$, the vertices are labeled as the $n$ variables, and there is a directed edge in $E$ $x_i\rightarrow x_j$ iff $x_i\in supp(f_j)$. 
\end{definition}

Monomials in the polynomial ring $\mathbb{F}_p[x_1,\ldots ,x_n]$ are written as  $x^\alpha=x_1^{\alpha_1}x_2^{\alpha_2}\cdots x_n^{\alpha_n}$, with exponent vector $\alpha=(\alpha_1,\ldots ,\alpha_n)\in \mathbb Z^n$. A \emph{monomial ideal} $I\subseteq \mathbb{F}_p[x_1,\dots,x_n]$ is an ideal generated by monomials, written as $I=\langle x^{\alpha},x^{\beta},\dots\rangle$. A monomial $x^\alpha$ is \emph{square free} if each $\alpha_i\in\{0,1\}$. A monomial ideal is a \emph{Stanley-Reisner ideal} if it can be generated by square-free monomials. 

A \emph{simplicial complex} over a finite set $X$ is a collection $\Delta$ of subsets of $X$ that are closed under the operation of taking subsets. That is, if $\beta\in\Delta$ and $\alpha\subseteq\beta$, then $\alpha\in\Delta$. The elements in $\Delta$ are called \emph{simplices} or \emph{faces}. Given an ideal $I$, we define the simplicial complex 
	\[
    \Delta_{I^c}=\{\alpha\mid x^\alpha\not\in I\},
    \]
 and given a simplicial complex $\Delta$ on $X=[n]=\{1,\ldots,n\}$, we define the square-free monomial ideal  
\[	
I_{\Delta^C}=\<x^\alpha\mid\alpha\not\in\Delta\>,
\]
which is the Stanley-Reisner ideal of $\Delta$. 

Consider a set  $V=\{\mathbf{s_1},\ldots,\mathbf{s_m}\}\subseteq \F_p^n$ of distinct input vectors, and a multiset $T=\{t_1,\ldots, t_m\}$ of output values from $\F_p$. We call
\[
\mathcal{D}= \{(\mathbf{s_1},t_1),\dots,(\mathbf{s_m},t_m)\}\subseteq\F^n\times\F
\]
the \emph{input-output data set}, where inputs may be stimuli applied to the network and outputs are its responses. A function $f\colon\F_p^n\to\F_p$ is said to \emph{fit the data} if $f(\mathbf{s})=t$ for all $(\mathbf{s},t)\in\mathcal{D}$. 
The \emph{model space} of $\mathcal{D}$ is the set of all functions that fit the data, i.e.
\[
\Mod(\mathcal{D})=\{f\colon\mathbb{F}_p^n\to\mathbb{F}_p^n\mid f(\mathbf{s})=t,\;\text{for all }(\mathbf{s},t)\in\mathcal{D}\}.
\]

For ease of presentation, we will focus on the wiring diagram of an individual node~$x_i$, that is, the edge set of the graph will be $E_{x_i}=\{(t,x_i)\mid t\in \supp(f_i)\}$. The union of the wiring diagrams of all nodes is, of course, the entire wiring diagram $W$. 

In \cite{JLSS}, the authors developed an algorithm for constructing all 
wiring diagrams based on sets of input-output data.  The method encoded certain coordinate changes  in input data as square-free monomials, generated a monomial ideal from these monomials, and used Stanley-Reisner theory to  decompose the ideal into primary components. These primary components were named \emph{minimal sets} or \emph{minsets} for short. A minset is a  set $S$ of variables so that there is a function in terms of those variables that fits the given data and there is no such function on proper subsets of $S$ (a formal definition will be presented as Definition~\ref{minset}). A 
wiring diagram for a specific node~$x$ can be constructed by drawing edges from the variables in the minset towards $x$. Details are provided in the following definitions and results from~\cite{JLSS}.

For every pair of distinct input vectors $\s=(s_1,\dots,s_n)$ and $\s'=(s'_1,\dots,s'_n)$ in $V$, we can encode the coordinates in which they differ by a square-free monomial 
\[
m(\s,\s')=\prod_{s_i\neq s'_i}x_i.
\]
Let $\M(V)$ be the set of all such monomials from $V$, that is,
\begin{equation}\label{sf-mon}
\M(V)=\{m(\s,\s')\mid\s,\s'\in V,\;\s\neq\s'\}.
\end{equation}
If distinct input vectors $\s,\s'\in V$ have different output values, $t\neq t'$, then any function $f\colon\F_p^n\to\F_p$ satisfying $f(\s)=t$ and $f(\s')=t'$ must depend on at least one of the variables in $m(\s,\s')$. In this case, we say that the support of $m(\s,\s')$, i.e., the set of variables that appear in it, is a \emph{non-disposable set} of $\D$. 

For a fixed data set $\D$, the non-disposable sets in the power set $2^{[n]}$, where $[n]=\{1,\ldots, n\}$,  are clearly closed under unions. We call all other sets \emph{disposable}, i.e. $\alpha\subseteq [n]$ is a disposable set of $\D$ if and only if there is some $f\in\Mod(\D)$ that depends only on the variables \emph{not} in~$\alpha$. Equivalently, its support satisfies
$supp(f)\subseteq\overline{\alpha}=[n]\setminus\alpha$. It is easy to see that disposable sets are closed under intersections. As such we can define the abstract \emph{simplicial complex of disposable sets} of~$\D$ to be
\[
\Delta_\D=\{\alpha\subseteq[n]\mid \alpha\text{ is a disposable set of } \D\}.
\]
If we canonically identify square-free monomials with subsets of $[n]$, then the \emph{Alexander dual} of~$\Delta_\D$ is the Stanley-Reisner ideal 
\[
I_{\Delta^c_\D}=\<x^\alpha\mid \alpha\not\in\Delta_\D\>=\<m(\s,\s')\mid t\neq t'\>,
\]
which is called the \emph{ideal of non-disposable sets}. By the Alexander duality, the simplicial complex of disposable sets is 
\[
\Delta_\D=\{\alpha\subseteq[n]\mid \alpha\not\in I_{\Delta^c_\D}\}.
\]

Since $I_{\Delta^c_\D}$ is squarefree, it has a unique primary decomposition, where the primary components are prime ideals generated by the variables in the complements of the facets (maximal faces) of $\Delta_\D$ (i.e., complements of maximal disposable sets). For a facet $\alpha\subseteq[n]$, denote the corresponding  primary component by $\p^{\overline{\alpha}}$. For example, if $n=5$ and $\alpha=x_2x_5$, then $\p^{\overline{\alpha}}=\<x_1,x_3,x_4\>$. The primary decomposition is thus 
\[
I_{\Delta^c_\D}=\bigcap_{\alpha\in\Delta_\D}\p^{\overline\alpha}=\bigcap_{\substack{\alpha\in\Delta_\D \\ \alpha\text{ maximal}}}\p^{\overline\alpha}.
\]
Over a field, being prime and being primary are equivalent properties for square-free monomial ideals. The ideal $I_{\Delta^c_\D}$ is prime if and only if it has only one primary component, which means that there is a unique maximal disposable set (i.e., a facet) $\alpha$ in $\Delta_\D$, and so
\[
I_{\Delta_\D^c}=\p^{\overline{\alpha}}=\<x_i\mid i\not\in\alpha\>.
\]
Thus the set $\G=\{x_i\mid i\not\in\alpha\}$ is a Gr\"obner basis for $I_{\Delta_\D^c}$. The converse holds as well: if a reduced Gr\"obner basis for $I_{\Delta_\D^c}$ has only single-variable monomials, then it must be prime. We summarize this next. 
\begin{theorem}\label{equiv}
The simplicial complex of disposable sets $\Delta_\D$ has a unique facet if and only if the ideal of non-disposable sets $I_{\Delta_\D^c}$ is prime.
\end{theorem}

By the Alexander duality, the primary components of $I_{\Delta^c_\Delta}$ are in bijection with the complements of the maximal disposable sets. Such a complement $\overline{\alpha}$ is precisely a minimal subset of~$[n]$ on which a function in the model space $\Mod(\D)$ can depend. This motivates the following definition. 

\begin{definition}[\cite{JLSS}]\label{minset}
The complement $\overline{\alpha}$ of a maximal disposable set $\alpha$ in $\Delta_\D$ is called a \emph{minimal set}, or \emph{minset} for short. 
\end{definition}

Each minset is a set of variables on which a polynomial can depend based on the data, and one that is minimal with respect to inclusion. These variables also encode the wiring diagram of a minimal number of edges incident to the node under consideration. We call such wiring diagrams \emph{minimal} as well.

We will use the following tumor-suppression network mediated by epidermal derived growth factor receptor (EGFR) \cite{Steinway2016} as a running example.
We consider Boolean and non-Boolean data for the gene network of three parameters (EGFR, Rasgap, and miR221) and three variables (Rkip, Kras, and Raf1), and an outcome of this network is proliferation or suppression of a tumor.  For illustration purposes, we focus on identifying the direct regulators of Raf1 from among the candidates Rasgap, Rkip, and Kras.   %

\begin{example}\label{unsigned-egfr}
Suppose we want to determine which nodes Raf1 depends on -- Rasgap, Rkip, or Kras -- based solely on experimental data. Suppose experiments are performed to generate the following input-output data (parentheses and commas are suppressed for readability):
$$\mathcal{D}=\{(\s_1,t_1),(\s_2,t_2),(\s_3,t_3),(\s_4,t_4)\}
=\{(000,1),(101,1),(110,0),(011,1)\},$$
where $\s_i=(Rasgap, Rkip, Kras)=(x_1,x_2,x_3)$ and $t_i$ is the corresponding value of Raf1.
That is, we want to determine the minimal sets of variables that appear in the unknown function $f\colon\F_2^3\to\F_2$
which determines the behavior of Raf1 based on input from the other three nodes, and fits the experimental data, that is,
$f(000)=1,f(101)=1$, $f(110)=0$, and $f(011)=1$.

Since $t_1=t_2=t_4\neq t_3$, we compute $m(\s_1,\s_3)
=x_1x_2$, $m(\s_2,\s_3)=x_2x_3$, and $m(\s_3,\s_4)
=x_1x_3$.
The ideal of non-disposable sets is thus $I_{\Delta_\D^c}=\<x_1x_2, x_2x_3, x_1x_3\>$ and has primary decomposition $\<x_1, x_2\>\cap\<x_1,x_3\>\cap\<x_2,x_3\>$, corresponding to these minimal wiring diagrams:

\[
  \tikzstyle{v} = [draw,inner sep=0pt, minimum size=3mm] 
  \tikzstyle{activ} = [draw, -stealth]
 \begin{tikzpicture}[scale=1]
      \node (1) at (0,2) {\small Rasgap};
      \node (2) at (1,2) {\small Rkip};
      \node (3) at (2,2) {\small Kras};
      \node (i) at (1,0) {\small Raf1};
      \draw [activ] (1) to[bend right,shorten >= 2pt] (i);
      \draw [activ] (2) to[shorten >= 2pt] (i);
 \end{tikzpicture} \hspace{8mm}
 \begin{tikzpicture}[scale=1]
      \node (1) at (0,2) {\small Rasgap};
      \node (2) at (1,2) {\small Rkip};
      \node (3) at (2,2) {\small Kras};
      \node (i) at (1,0) {\small Raf1};
      \draw [activ] (1) to[bend right,shorten >= 2pt] (i);
      \draw [activ] (3) to[bend left,shorten >= 2pt] (i);
 \end{tikzpicture} \hspace{8mm}
 \begin{tikzpicture}[scale=1] 
      \node (1) at (0,2) {\small Rasgap};
      \node (2) at (1,2) {\small Rkip};
      \node (3) at (2,2) {\small Kras};
      \node (i) at (1,0) {\small Raf1};
      \draw [activ] (2) to[shorten >= 2pt] (i);
      \draw [activ] (3) to[bend left,shorten >= 2pt] (i);
\end{tikzpicture} 
\]
\end{example}

The  limited information that these experimental data support is that any two of the three nodes can influence Raf1. If, in addition, we perform an experiment where the input nodes are all expressed and Raf1 happens to also be expressed as a result, this will add to $\mathcal{D}$ the data point $(\s_5,t_5)=(111,1)$. As a result, the monomial $x_2$  will be added to $I_{\Delta_\D^c}$ whose primary decomposition now becomes $\<x_1, x_2\>\cap\<x_2,x_3\>$, eliminating the middle wiring diagram from the figure above. Since $x_2$ is in both primary ideals, we are now confident that Rkip affects Raf1. While we still do not know if Rasgap and Kras participate in the regulation of Raf1, this may be sufficient if the role of Rkip is the focus of the experimental work.

On the other hand, if instead of adding $(\s_5,t_5)=(111,1)$, we added $(\s'_5,t'_5)=(010,0)$, the new monomials added to $I_{\Delta_\D^c}$ will be not only $x_2$ but also $x_1x_2x_3$ and $x_3$. Now the primary decomposition becomes $\<x_2,x_3\>$, reducing the possible wiring diagrams to a unique one (rightmost above) and completely determining the regulation of Raf1.

This example shows that some input-output data sets result in multiple models, whereas well-chosen datasets can reduce the number of possible wiring diagrams and even lead to a unique model.

\section{Main results}
\label{sec:main}




\subsection{Theoretical Results}

The one-to-one correspondence between the minsets of $\Delta_{\mathcal{D}}$ and the minimal wiring diagrams of $\Mod(\mathcal{D})$ implies that finding input sets which uniquely identify the minimal wiring diagram underlying a system is equivalent to finding input sets $\mathcal{D}$ whose corresponding simplicial complexes $\Delta_{\mathcal{D}}$ have a unique minset.
The theory of minsets developed in \cite{JLSS, veliz-cuba-signed-ms} establishes methods for generating all minimal wiring diagrams for a given input-output data set $\mathcal{D}$. In practice, however, one does not know the experimental output~$T$ \textit{a priori}. Therefore, it is desirable to develop theory and algorithms which allow us to design experiments whose output is guaranteed to reduce the size of the wiring-diagram space of the system without making assumptions for the unknown experimental outcome. In the next section, we provide  necessary and sufficient conditions on the input data set which are computationally feasible to guarantee that the identified minset is unique regardless of the output. 

\subsubsection{Identifying input sets corresponding to a unique minset}






Based on Theorem~\ref{equiv}, our goal is to efficiently identify sets whose ideal of non-disposable sets in prime. Below we construct an algorithm for the identification of such input sets.

Let $V=\{\mathbf{s}_1,\ldots,\mathbf{s}_m\}\subseteq \F_p^n$ be an input set of distinct vectors.  We define the multiset 
\begin{equation}\label{pairs}
M=\left\{m(\s_i,\s_j)\mid  i,j\in [r],\;1\leq i<j\leq r\right\},
\end{equation}
where ${\displaystyle m(\s_i,\s_j)=\prod_{\s_{ik}\neq \s_{jk}}x_i}$ are square-free monomials which record the coordinates where each pair of points in $V$ differ. The number of pairs in this set is $|M|=(r-1)+(r-2)+\cdots +1=\frac{(r-1)r}{2}$ since, unlike in (\ref{sf-mon}), monomials are repeated if they come from different input pairs. 
For example, if $m(s_1,s_2)=m(s_2,s_6)=x_2x_5$, then $x_2x_5$ will be listed twice and it will be recorded to which input pairs it corresponds.
Let $M_{MV}$ be the list of multivariate monomials in $M$, again keeping track of which pairs of points in $V$ yielded each monomial.
For each $m(s_a,s_b)\in M_{MV}$, let $m(s_{i_1},s_{j_1}),\ldots,m(s_{i_k},s_{j_{\ell}})$ be the single-variate monomials in $M$ that divide $m(s_a,s_b)$. 

\begin{theorem}\label{unique-minset}
Let $V=\{\mathbf{s}_1,\ldots,\mathbf{s}_m\}\subseteq \F_p^n$ be an input set of distinct vectors, and $M$ and $M_{MV}$ be defined as above. Let $t_k$ denote the unknown output of $s_k$. There exists an output assignment $T$ for which $I_{\Delta^c}$ is not prime (and so there are multiple minsets) if and only if there is a monomial in $M_{MV}$ for which the following system is consistent. 

\begin{eqnarray}\label{sys}
   t_a&\ne &t_b\nonumber \\ 
   t_{i_1}&=&t_{j_1}\\
   & \vdots & \nonumber\\
   t_{i_k}&=&t_{j_{\ell}}\nonumber
\end{eqnarray}

\end{theorem}

\begin{proof}
The system is set up so that if consistent, $\mathcal M(V)$ from (\ref{sf-mon}) will contain at least one multivariate monomial without a single-variate monomial that divides it.  In that case, the primary decomposition of the ideal generated by the monomials in $\mathcal M(V)$ will have more than one primary component.
\end{proof}

Notice that the equations in (\ref{sys}) form a homogeneous linear system whose coefficient matrix is sparse and solving it is computationally easy. 
As soon as a consistent system is found for an element in $M_{MV}$, we can stop and conclude that there exists a $T$ for which $I_{\Delta^c}$ is not prime. If no such system is found, then for any $T$, the Gr\"obner basis of $I_{\Delta^c}$ consists entirely of single-variate monomials and so $I_{\Delta^c}$ is prime for all output assignments.

To illustrate the process above consider the following examples.

\begin{example}

Consider the following non-Boolean input data for the EGFR network in~\cite{Steinway2016}, where $\s_i=(Rasgap, Rkip, Kras)=(x_1,x_2,x_3)$ and $t_i$ is the corresponding value of Raf1: $V=\{\s_1,\s_2,\s_3,\s_4\}=\{(010),(110),(210),(212)\}\subseteq \F_3^3$. The set $M$ contains the monomials
    $m(\s_1,\s_2)=x_1, m(\s_1,\s_3)=x_1, m(\s_1,\s_4)=x_1x_3, m(\s_2,\s_3)=x_1, m(\s_2,\s_4)=x_1x_3,$ $m(\s_3,\s_4)=x_3$.
    The multivariate monomials are $m(\s_1,\s_4)=x_1x_3$ and $m(\s_2,\s_4)=x_1x_3$. The two corresponding systems below are both inconsistent and so $V$ has a unique minset for any~$T$. 

\begin{multicols}{2}    
\begin{eqnarray}
   t_1&\ne &t_4\nonumber\\
   t_1&=&t_2\nonumber\\
   t_1&=&t_3\nonumber\\
   t_2&=&t_3\nonumber\\
   t_3&=&t_4\nonumber
\end{eqnarray} 

\begin{eqnarray}
   t_2&\ne &t_4\nonumber\\
   t_1&=&t_2\nonumber\\
   t_1&=&t_3\nonumber\\
   t_2&=&t_3\nonumber\\
   t_3&=&t_4\nonumber
\end{eqnarray}
\end{multicols}

The algorithm determines that regardless of the experimental output, this input set $V$ is guaranteed to result in a unique minimal wiring diagram for Raf1. (Notice that while unique for any output, the wiring diagram will vary based on the output.)

Now consider the input data set $U=\{\s_1,\s_2,\s_3,\s_4\}=\{(211),(002),(200),(201)\}\subseteq \F_3^3$. The monomials in $M$ are
    $m(\s_1,\s_2)=x_1x_2x_3, m(\s_1,\s_3)=x_2x_3, m(\s_1,\s_4)=x_2, m(\s_2,\s_3)=x_1x_2, m(\s_2,\s_4)=x_1x_2,$ and $m(\s_3,\s_4)=x_3$. 
    Based on the multivariate monomial $m(\s_1,\s_2)=x_1x_2x_3$, we form the consistent system
    $$t_1\ne t_2, \ \ t_1=t_4, \ \ t_3=t_4.$$
The algorithm identifies that there exist output assignments for which 
$I_{\Delta^c}$ is not prime. For example, $T=\{0,2,0,0\}$, i.e. $t_1=t_3=t_4=0, t_2=2$, corresponds to two minsets: $\{x_2\}$ and $\{x_3\}$; that is, we can have experimental output that will result in two possible minimal wiring diagrams for Raf1: in one Raf1 depends on Rkip only, and in the other Raf1 depends on Kras only. 

\end{example}

Having built an algorithm for identifying if an input data set $V$ corresponds to a unique minset, we next ask how a unique minset  relates to the Gr\"obner basis of $I(V)$ and to the normal form of polynomials that take $V$ as input.

\subsubsection{Polynomial normal forms and minsets}

The main result in this section is Theorem~\ref{unique-nf-minset} which establishes that a unique normal form (regardless of monomial order) of a polynomial that fits a set of input-output pairs $\mathcal D$ implies a unique minset for $\mathcal D$.

\begin{lemma}[\cite{he-unique-gbs}]\label{lma-mon}
Let $x^{\alpha},x^{\beta}$ be monomials with $x^{\alpha} \nmid x^{\beta}$. There exists a weight vector $\gamma$ and monomial order $\prec_{\gamma}$ such that $x^{\beta} \prec_{\gamma} x^{\alpha}$.
\end{lemma}

\begin{proof}
Let $x^{\alpha} \nmid x^{\beta}$. As $x^{\alpha} \nmid x^{\beta}$, $\alpha_j  > \beta_j$ for some coordinate $j$.  Take $\gamma$ to be a vector in $\mathbb{R}^n$ with a sufficiently large rational value in entry $j$ and square roots of distinct prime numbers elsewhere such that $\gamma \cdot \alpha > \gamma \cdot \beta$.  Then the entries of $\gamma$ are linearly independent over $\mathbb{Q}$ and so~$\gamma$ defines a weight order. Define $\prec_\gamma$ to be the monomial order weighted by $\gamma$.  It follows that $x^\beta \prec_\gamma x^\alpha$. 
\end{proof}

\begin{theorem}\label{unique-nf-minset}
Let $\mathcal{D}\subseteq \mathbb{F}^n\times\mathbb{F}$ be a data set of input-output pairs and let  $f:\mathbb{F}^n\to \mathbb{F}$ be any polynomial that fits $\mathcal{D}$. If $f$ has a unique normal form for all  Gr\"obner bases of $I(V)$, then $\mathcal{D}$ has a unique minset.
\end{theorem}

\begin{proof} Let $\overline{f}$ be the unique normal form of $f$ with respect to $I(V)$. For contradiction, suppose that there exists a polynomial $h$ that fits $\mathcal{D}$ such that supp$(\overline{f})$ contains a variable~$x_i$ which is not in supp$(h)$. Notice that $\overline{f}-h\in I(V)$ and all monomials of $\overline{f}$ that contain~$x_i$ are in 
$\overline{f}-h$. Since a monomial that contains~$x_i$ does not divide a monomial that does not contain~$x_i$, it follows by Lemma~\ref{lma-mon} that there is a monomial order $\prec$ under which some monomial~$x^{\alpha}$ of $\overline{f}-h$  which contain $x_i$ is the leading monomial of $\overline{f}-h$ and thus it is in $in_{\prec}(I(V))$. This is a contradiction since~$x^{\alpha}$ is a monomial of~$\overline{f}$ which is a linear combination of monomials that are standard with respect to any monomial order as the normal form is unique.
\end{proof}

One consequence of the previous theorem is that the support of a unique normal form is a minset.  Another is the following key condition on inputs.




\begin{corollary}\label{ugb-minset}
Let $V$ be a set of inputs. If $I(V)$ has a unique Gr\"obner basis, then for all output assignments there is a unique minset.
\end{corollary}

Notice that the converse of Corollary~\ref{ugb-minset} is false. For example, $V=\{00, \allowbreak 10, \allowbreak 01, \allowbreak 11, \allowbreak 02, \allowbreak 20, \allowbreak 22\}\subseteq \mathbb{Z}_3^2$ has an ideal $I(V)$ with two Gr\"obner bases, $\{x+y,y^2-1\}$ and $\{x^2-1,y+x\}$, but~$V$ has only one minset for any output $T$.



Theorem~\ref{unique-nf-minset} and its corollaries beg the following question in algebraic design of experiments: What input-output data corresponds to a model with a unique normal form? We answer that in Theorem~\ref{unique-nf} below.

\begin{definition}
    Let $\lambda=\{u^1, \ldots, u^r\}$ be an $r$-subset of $\mathbb{N}^n_p$ and let $V=\{v^1, \ldots, v^s\}$ be an $s$-subset of $\mathbb{N}^n_p$. The \emph{evaluation matrix} $\mathbb{X}(x^{\lambda},V)$ is the $s$ by $r$ matrix whose element in position $(i,j)$ is $x^{u^j}(v^{i})$, the evaluation of $x^{u^j}$ at $v^{i}$.
\end{definition}

\begin{example}
Consider $V=\{(0,0,1), (0,1,0), (1,0,1)\}\subset \mathbb{F}_2^3$. One of its sets of standard monomials is $x^{\lambda}=\{1,z,x\}$ which corresponds to the set of exponent vectors $\lambda=\{(0,0,0), (0,0,1), (1,0,0)\}$ and produces the following evaluation matrix on $V$:
 \begin{center}
$\mathbb{X}(x^{\lambda},V)=\begin{bmatrix} 
1 & 1 & 0 \\
1 & 0 & 0 \\ 
1 & 1 & 1 \\
\end{bmatrix}.$
\end{center}
\end{example}

\begin{theorem}\label{unique-nf}
Let $\mathbb{F}$ be a field. Consider a set $V=\{s_1 \ldots, s_r\} \subseteq \mathbb{F}^n$ of distinct input vectors and an output vector $T=(t_1, \ldots, t_r )\in \mathbb{F}^r$. Let $f \in \mathbb{F}[x_1, \ldots, x_n]$ be such that $f(s_i)=t_i$ for all $i \in \{1, \ldots, r \}$. The normal form of $f$ is unique with respect to any Gr\"obner basis if and only if $T$ is a linear combination of the columns in $\mathbb{X}(x^{\lambda},V)$ that correspond to monomials which are standard with respect to any Gr\"obner basis. 
\end{theorem}

\begin{proof} First suppose that
$T$ is a linear combination of the columns 
of $\mathbb{X}(x^{\lambda},V)$ which correspond to the standard monomials in the intersection of all sets of standard monomials. Therefore, the normal form of the interpolating polynomial $f$ is also a linear combination (with the same coefficients) of standard monomials that appear in every set of standard monomials and so will not change as we change the Gr\"obner basis.

Conversely, suppose that the normal form of $f$ is unique with respect to any Gr\"obner basis. Then the normal form of $f$ is a linear combination of monomials that are standard with respect to any Gr\"obner basis and so $T$ is (the same) linear combination of the columns in the evaluation matrix that correspond to the monomials that are standard for every Gr\"obner basis.
\end{proof}

\begin{example}
Consider an input set $V = \{(0,0,1), (0,1,1), (1,0,1), (1,1,0)\} \subset \mathbb{F}_2^{3}$. $I(V)$ has exactly two distinct sets of standard monomials, namely $SM_1 = \{1,z,y,x\}$ and $SM_2 = \{1,y,x,xy\}$, resulting from different monomial orderings, with $SM_1 \cap SM_2 = \{1,x,y\}$. The evaluation matrices for each of these sets of standard monomials are
\vspace{5pt}
\begin{center}
$\begin{matrix}
SM_1 & \vline & 1 & z & y & x \\ \hline
(0,0,1) & \vline & 1 & 1 & 0 & 0 \\
(0,1,1) & \vline & 1 & 1 & 1 & 0 \\
(1,0,1) & \vline & 1 & 1 & 0 & 1 \\
(1,1,0) & \vline & 1 & 0 & 1 & 1 \\
\end{matrix}$ \\
\end{center}
and 
\begin{center}
$\begin{matrix}
SM_2 & \vline & 1 & y & x & xy \\ \hline
(0,0,1) & \vline & 1 & 0 & 0 & 0 \\
(0,1,1) & \vline & 1 & 1 & 0 & 0 \\
(1,0,1) & \vline & 1 & 0 & 1 & 0 \\
(1,1,0) & \vline & 1 & 1 & 1 & 1 \\
\end{matrix}$
\end{center}
Take, for example, the sum of the matrix columns that are the evaluations of the monomials in $SM_1 \cap SM_2 = \{1,x,y\}$: $[1, 0, 0, 1]^{T}$, i.e. one linear combination. We find a polynomial function $f \in \mathbb{F}_2[x,y,z]$ that maps each input point to the corresponding output value as follows:
\begin{center}
$\begin{matrix}
    (0,0,1) & \mapsto & 1\\
    (0,1,1) & \mapsto & 0\\
    (1,0,1) & \mapsto & 0\\
    (1,1,0) & \mapsto & 1\\
\end{matrix}$
\end{center}
We find such $f$ via, say, Lagrange interpolation, to be $f=xy+xz+yz+z$. Now we compute the normal forms of $f$ reduced by $G_1$ and $G_2$, where $G_1$ and $G_2$ are the Gröbner bases for the ideal $I(V)$ corresponding to $SM_1$ and $SM_2$, arriving at
\begin{center}
    $\overline{f}^{G_1} = \overline{f}^{G_2}= x+y+1$.
\end{center}
Since $G_1$ and $G_2$ are the only two reduced  Gröbner bases for the ideal, this normal form is unique. 

If, instead, we take the same input set $V$ and corresponding standard monomials but choose a new output vector, one that is not a linear combination of the columns corresponding to monomials that are standard with respect to any monomial ordering, we expect to find more than one distinct normal form of $f$. Consider, for example, the output vector $[0,1,1,1]^{T}$. That is, we are looking for a polynomial function $f \in  \mathbb{F}_2[x,y,z]$ which maps
\begin{center}
$\begin{matrix}
    (0,0,1) & \mapsto & 0\\
    (0,1,1) & \mapsto & 1\\
    (1,0,1) & \mapsto & 1\\
    (1,1,0) & \mapsto & 1\\
\end{matrix}$
\end{center}
This time, $f$ has two distinct normal forms,
    $$\overline{f}^{G_1}=x+y+z+1 ~~ \textrm{and}~~ \overline{f}^{G_2}=xy+x+y.$$
So as expected, $\overline{f}^{G_1} \neq \overline{f}^{G_2}$. 
\end{example}


\begin{corollary}\label{nf-ugb}
The normal form of $f \in \mathbb{F}[x_1, \ldots, x_n]$ that fits a data set with input $V=\{s_1 \ldots s_r\} \subseteq \mathbb{F}^n$ is unique for any output $T$ if and only if $I(V)$ has a unique reduced Gr\"obner basis.
\end{corollary}


Corollaries~\ref{ugb-minset} and~\ref{nf-ugb} point towards the importance of ideals $I(V)$ that have a unique reduced Gr\"obner basis. Such ideals were studied in~\cite{he-unique-gbs, robbiano}, where sufficient conditions for $I(V)$ to have a unique reduced Gr\"obner basis were introduced; in this paper, Corollary~\ref{ugb-diag-free} is a necessary condition that depends on a special relation between the points in $V$ that we define next. 


\begin{definition}
A pair of points \(p,q \in \mathbb{F}_p^n\) form a \emph{diagonal} if $p$ and $q$ differ in at least two coordinates. We will also say that a set $V$ \emph{contains a diagonal} if there is a point $p\in V$ which forms a diagonal with all other points in $V$.
\end{definition}

\begin{theorem}\label{diagonal}
If $V$ contains a diagonal, then there exists an output assignment that corresponds to multiple minsets.
\end{theorem}

\textbf{Proof:} Let $p\in V$ form a diagonal with all other points in $V$. Then there is a point $s\in V$ such that $m(p, s)$ is a multivariate monomial. Denote the corresponding outputs from $p$ and~$s$ by $t_p$ and $t_s$.

\begin{itemize}
\item[Case 1:] There are no points $s_i, s_j\in V$ such that $m(s_i, s_j)$ is a single-variate monomial that divides $m(p, s)$. Then according to Theorem~\ref{unique-minset} there is an output assignment for which there are multiple minsets.

\item[Case 2:] There are pairs of points $s_i, s_j\in V$ for which $m(s_i, s_j)$ is a single-variate monomial that divides $m(p, s)$. However, since $m(p, s_i)$ and $m(p, s_j)$ are multivariate, we know that $p\ne s_i$ and $p\ne s_j$. Therefore, one can choose an output assignment $T$ where $t_i=t_j$ for all pairs of input points $s_i, s_j\in V$ such that $m(s_i, s_j)$ is a single-variate monomial that divides $m(p, s)$, while also choosing $t_p\ne t_s$. According to Theorem~\ref{unique-minset}, there are multiple minsets for this $T$.
\end{itemize}

\begin{corollary}\label{ugb-diag-free}
If $I(V)$ has a unique reduced Gr\"obner basis, then $V$ is diagonal-free.
\end{corollary}
\begin{proof} The contrapositive of Theorem~\ref{diagonal} is ``If $V$ corresponds to a unique minset for any output assignment, then $V$ is diagonal-free.'' which follows from Corollary~\ref{ugb-minset}.
\end{proof}

\section{Experimental results}
\label{sec:experiments}

%
%
%
%
%

Theorem~\ref{diagonal} suggests the following heuristic idea that we will test computationally: \textit{the smaller the Hamming distance between points in $V$, the smaller the number of minsets}. To quantify the Hamming distance between points in $V$, we use the following definition. 

\begin{definition}
Given an input set $V$, we define $d(V)$ as the average value of the Hamming distance $H(p,q)$  between distinct points~$p$ and~$q$ of~$V$. We call $d(V)$ the internal distance.
\end{definition}

\begin{example}\label{ex:small_experiments}

Consider $f:\F_2^3\rightarrow \F_2$ given by $f(x_1,x_2,x_3)=\Not{x_2} \vee x_1$ or equivalently, $f(x_1,x_2,x_3)=1+x_2+x_2x_3$. To illustrate the definition we consider two different input sets, $V_1=\{000,001,010,100\}$ and $V_2=\{000,101,110,011\}$.

The distance between points in $V_1$ is given below.
\begin{center}
$\begin{matrix}
(\mathbf{s}_1,\mathbf{s}_2)  & H(\mathbf{s}_1,\mathbf{s}_2) \\ \hline
(000,001) & 1 \\
(000,010) & 1  \\
(000,100) & 1 \\
(001,010) & 2  \\
(001,100) & 2  \\
(010,100) & 2  \\
\hline
& d(V_1) = 1.5
\end{matrix}$ \\
\end{center}

Similarly,  $d(V_2)=2$. Now, we use $f$ to generate data sets for $V_1$ and $V_2$:
$\mathcal{D}_1=\{(000,1),\allowbreak (001,1),\allowbreak (010,0),\allowbreak (100,1)\}$
and $\mathcal{D}_2=\{(000,1),\allowbreak (101,1),\allowbreak (110,0),\allowbreak (011,1)\}$. $\mathcal{D}_1$ has the unique minset $\{x_2\}$ and $\mathcal{D}_2$ has the minsets $\{x_1,x_2\}$, $\{x_1,x_3\}$, $\{x_2,x_3\}$.

In summary, $V_1$ has an internal distance of $d(V_1)=1.5$ and resulted in $\#M(V_1)=1$ minset. $V_2$ has an internal distance of $d(V_2)=2$ and resulted in $\#M(V_2)=3$ minsets. 

The following table shows the statistics of all possible input sets with 4 points (there are $\binom{2^3}{4}=70$ of them). Some of them have the same  internal distance and/or number of minsets. This is reported in the following table and a scatter plot is shown in Figure~\ref{fig:scatter_plot_example}.

\begin{table}[h]
    \centering
    \begin{tabular}{c|c|c}
$d(V)$  & $\#M(V)$ & \text{ number of such $V$'s } \\ \hline
   1.3 &  1 & 6\\
   1.5 &  1 & 8\\
   1.7 &  1 & 24\\
   1.8 &  1 & 10\\
   1.8 &  2 & 14\\
   2 &  1 & 1\\
   2 &  2 & 5\\
   2 &  3 & 2\\
\hline
    \end{tabular}
    \caption{Statistics of all 70 possible $V$'s with 4 elements, grouped by internal distance and number of minsets.}
    \label{tab:distance_numminsets_example}
\end{table}

\begin{figure}[ht]
\centering
\includegraphics[width=4in]{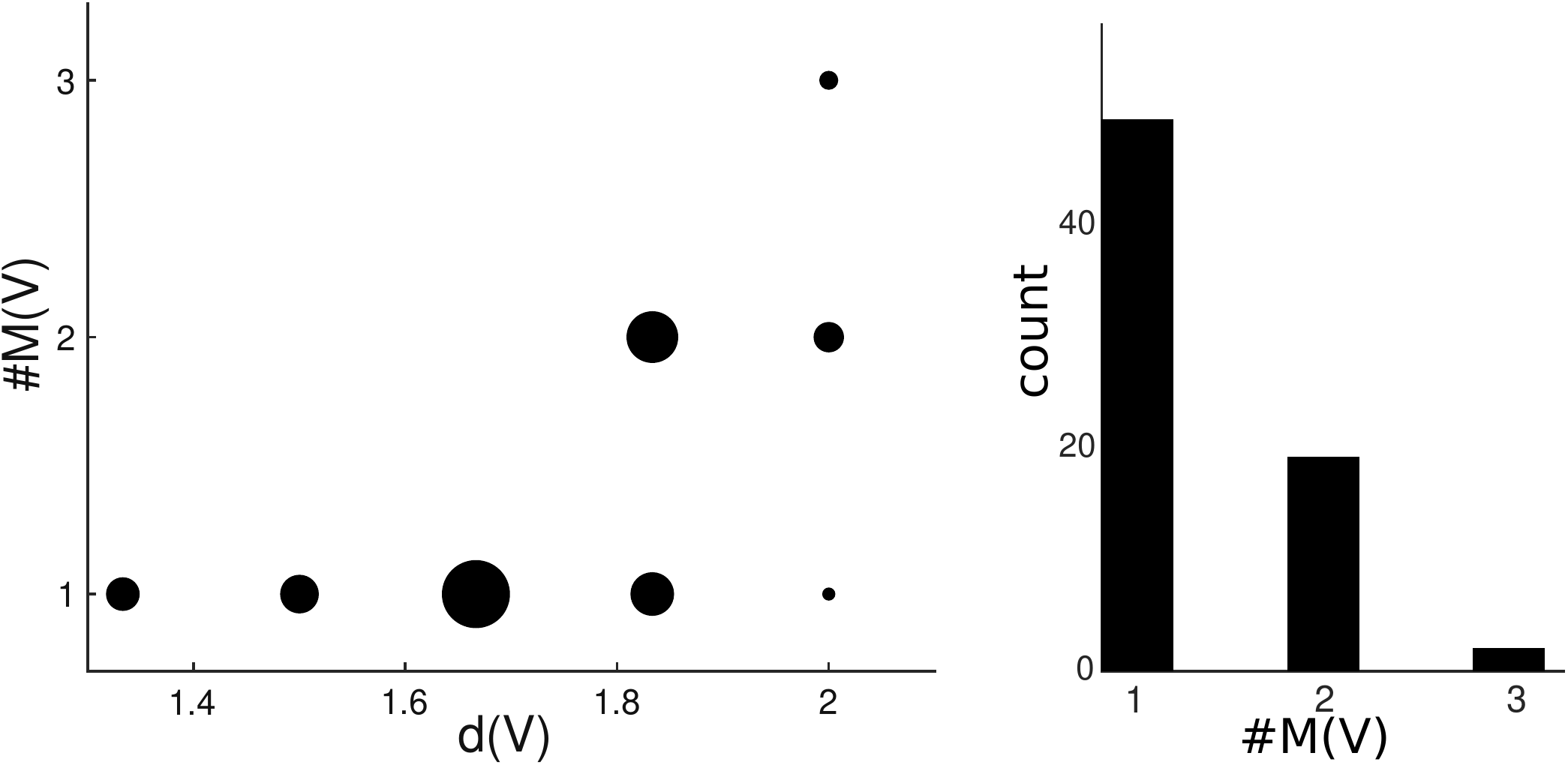}
\caption{Scatter plot of $\#M(V)$ vs $d(V)$ and histogram of $\#M(V)$ for \textbf{all} input sets with 4 points. The area of each circle corresponds to the number of $V$'s that have the same values of $d(V)$ and $\#M(V)$. We can see that as the internal distance increases, the number of minsets can get larger.}
\label{fig:scatter_plot_example}
\end{figure}
\end{example}

The results from Figure~\ref{fig:scatter_plot_example} are consistent with the heuristic idea that the smaller the distance, the smaller the number of minsets.
Now we would like to test two different strategies for generating data, one of which will tend to have small internal distance.

Consider a Boolean function $f:\F_2^n\rightarrow \F_2$. A \textit{trial} will consist of selecting an input set with~$m$ elements, $V\subseteq \F_2^n$. Then, we consider the data set $\mathcal{D}=\{(s,f(s)): s\in V\}$ and compute the minsets $M$. We are interested in the relationship between the internal distance of $V$, $d(V)$, and the number of minsets $\#M(V)$. If we plot the points $(d(V),\#M(V))$ for several trials, we expect to see some type of relationship like in Figure~\ref{fig:scatter_plot_example}. 
We used two different strategies or sampling schemes to generate the $m$ points in $V$.

\begin{itemize}
    \item 
    Pick $m$ points randomly. We refer to this as the \textit{random scheme}.
    
    \item
    Generate $m/2$ points randomly. Then, for each of those points, select a random entry to switch it. We refer to this as the \textit{small-distance scheme}.
\end{itemize}

In both cases we get an input set $V$ with $m$ points, but the small-distance scheme favours a smaller internal distance. 

The Boolean functions we selected for our analysis were \textit{fanout-free} (that is, each variable appears only once in its Boolean representation). These functions cover the vast majority of functions used in modeling \cite{mendoza2006method, sridharan2012boolean,mbodj2013logical,orlando2008global,veliz2011boolean,  giacomantonio2010boolean, helikar2015integrating, jenkins2017bistability, mbodj2013logical}. To keep the simulations tractable, we used Boolean functions $f:\F_2^{10}:\rightarrow\F_2$ such that $|supp(f)|\leq 4$. Up to a relabeling of variables and states, there are 9 such functions (not counting the constant functions), given in Table~\ref{tab:all_functions}.

\begin{table}[h]
    \centering
    \begin{tabular}{l|l}
    \textbf{Function in polynomial form} & \textbf{Function in Boolean form}\\ \hline
$x_1$ & $x_1$\\
$x_1 x_2$ & $x_1\wedge x_2$\\
$x_1 x_2 x_3$ & $x_1\wedge x_2\wedge x_3$\\
$x_1 (x_2 + x_3 + x_2 x_3)$ & $x_1\wedge (x_2 \vee x_3)$\\ 
$x_1 x_2 x_3 x_4$ & $x_1\wedge x_2\wedge x_3\wedge x_4$ \\
$x_1x_2x_3 + x_4 +x_1x_2x_3x_4$ & $(x_1\wedge x_2 \wedge x_3)\vee x_4$ \\
$x_1x_2x_3x_4 + x_1x_2x_3 + x_1x_2x_4 + x_1x_2 + x_3x_4 + x_3 + x_4$ & $(x_1\wedge x_2)\vee x_3 \vee x_4$ \\
$x_1x_2+x_3x_4+x_1x_2x_3x_4$ & $(x_1\wedge x_2)\vee (x_3 \wedge x_4)$ \\
$(x_1x_2+x_3+x_1x_2x_3 )x_4$ & $((x_1\wedge x_2)\vee x_3)\wedge x_4$
    \end{tabular}
    \caption{Boolean functions used for the computational analysis in Figure~\ref{fig:scatter_plot_9fun}. These represent all fanout-free functions with up to four variables.
    }
    \label{tab:all_functions}
\end{table}


The results of the simulations are shown in Figure~\ref{fig:scatter_plot_9fun}. As expected, the internal distance $d(V)$ is smaller when points are generated using the small-distance scheme (blue). Importantly, in the small-distance and random schemes, the smaller the internal distance, the smaller the number of minsets.  The histograms compare the number of minsets for both schemes and clearly show that the small-distance scheme results in a smaller number of minsets.
These computational results provide a straightforward way to generate data with a small number of minsets, the small-distance scheme.

\begin{figure}[ht]
\centering
\includegraphics[width=3in]{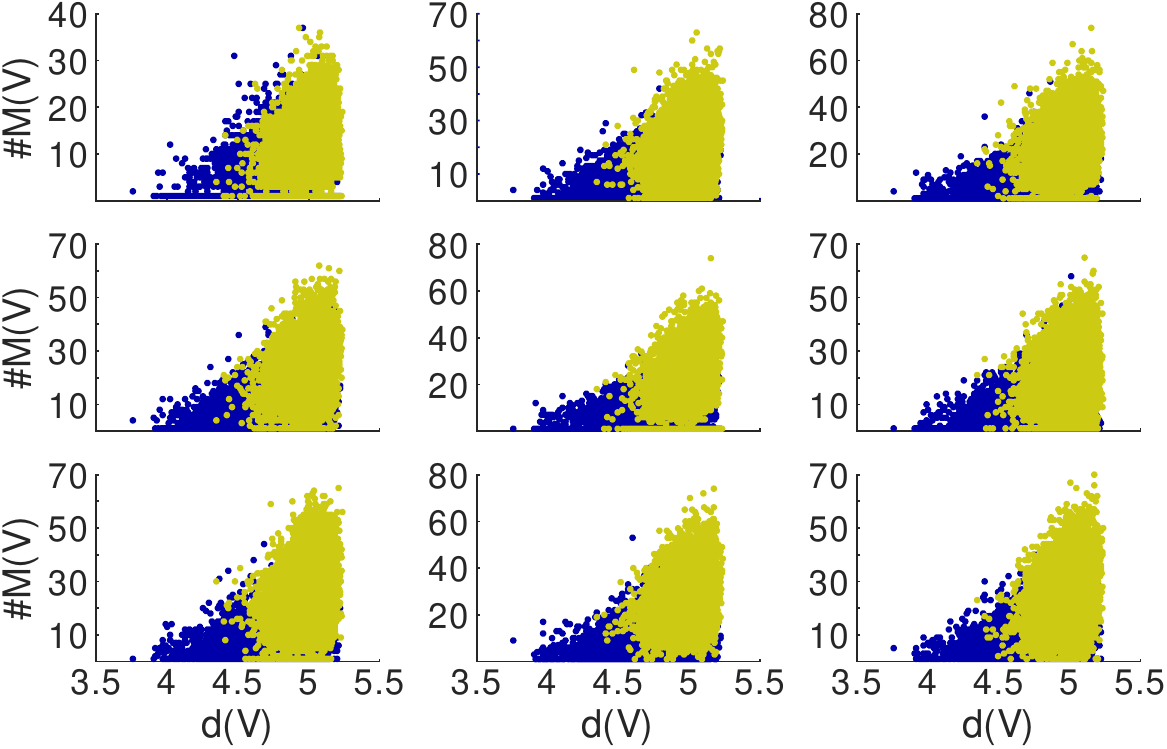}
\includegraphics[width=3in]{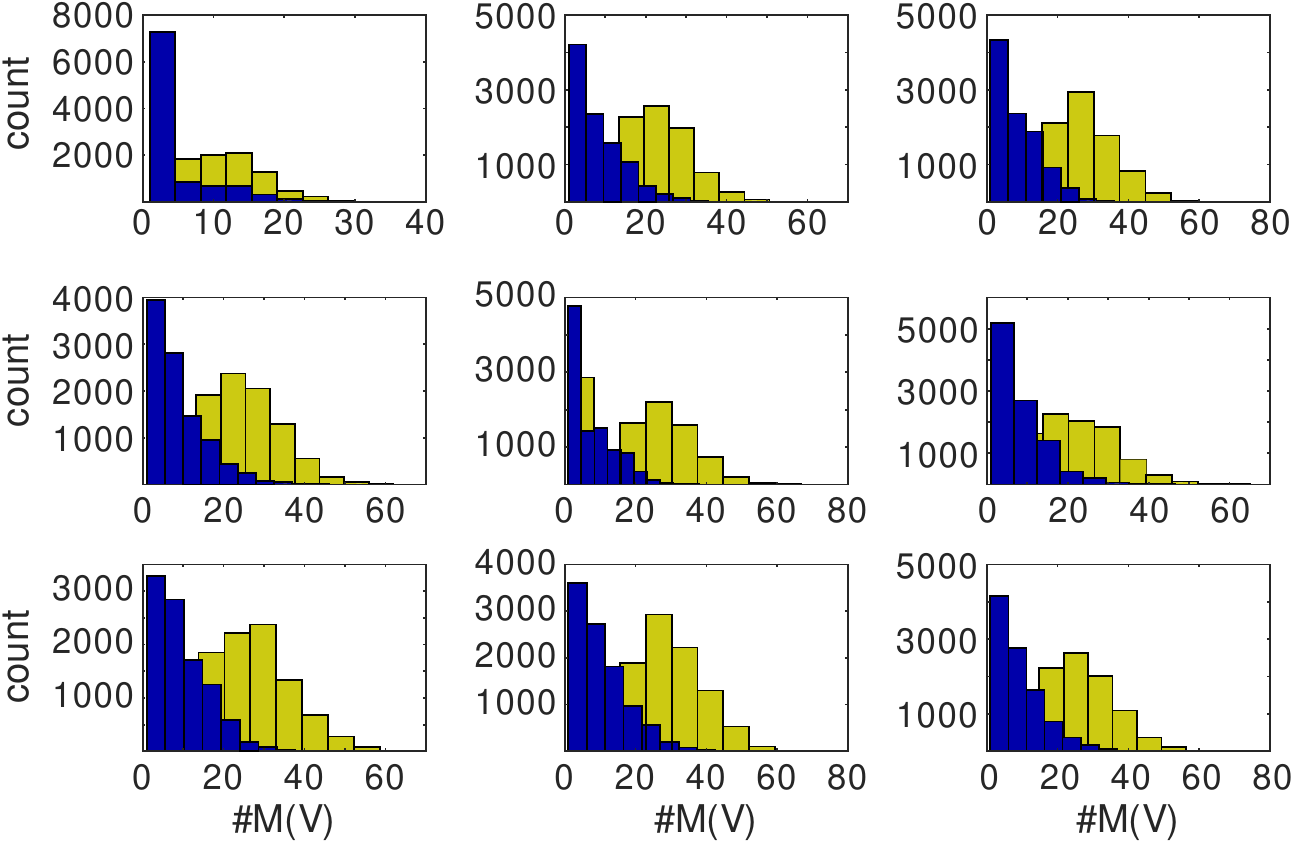}
\caption{
Scatter plots of $\#M(V)$ vs $d(V)$ and histograms of $\#M(V)$ for the functions in Table~\ref{tab:all_functions}. The scatter plots show that as the internal distance increases, the number of minsets can get larger (blue: small-distance scheme, yellow: random scheme). The histograms show that the small-distance scheme results in an overall smaller number of minsets. For each of the Boolean functions we run 10,000 trials with input sets with $m=20$ elements (about $2\%$ of the $2^{10}$ possible points). 
}
\label{fig:scatter_plot_9fun}
\end{figure}

\section{Conclusions and future work}
\label{sec:conclusions}
One of the difficulties in data-driven approaches is that there is typically a large number of models that fit the collected data and the known constraints of the system are not sufficient to reduce the pool of candidate models to a manageable size for testing and validation purposes.  As each model contains a set of predictions about the network being studied, even small numbers of competing models result in a combinatorial growth in validation experiments to be performed. Thus it is desirable to design experiments in such a way that maximizes the chance that the outputs will increase our understanding of the system. We introduced a method which generates data sets that are guaranteed to result in a unique minimal 
wiring diagram regardless of what the experimental outputs are. A natural next step is to extend these results to signed minimal wiring diagrams and address the question of existence for this case. The somewhat surprising connection between uniqueness of interpolating polynomial normal forms and unique minsets (i.e. unique minimal wiring diagrams) elucidate the role of polynomial ideals with unique Gr\"obner bases. While partial results are available in our prior work and in this manuscript, a complete geometric or combinatorial characterization of sets $V\subset \mathbb F_p^n$ such that $I(V)$ has a unique reduced Gr\"obner basis is still an open question whose importance has been emphasized in this work.




\bibliographystyle{siamplain}
\bibliography{refs}

\end{document}